\documentclass[11pt,a4paper]{amsart}
\usepackage{hyperref}
\usepackage{amsfonts}
\usepackage{amsthm}
\usepackage{amsmath}
\usepackage{amscd}
\usepackage[latin2]{inputenc}
\usepackage{t1enc}
\usepackage[mathscr]{eucal}
\usepackage{indentfirst}
\usepackage{graphicx}
\usepackage{graphics}
\usepackage{pict2e}
\usepackage{epic}
\numberwithin{equation}{section}
\usepackage[margin=2.9cm]{geometry}
\usepackage{epstopdf} 
\usepackage{color}

\theoremstyle{plain}
\newtheorem{thm}{Theorem}[section]
\newtheorem{lem}[thm]{Lemma}
\newtheorem{cor}[thm]{Corollary}
\newtheorem{prop}[thm]{Proposition}
\newtheorem*{thm*}{Theorem}
\newtheorem*{prop*}{Proposition}
\theoremstyle{definition}

\newtheorem{rem}[thm]{Remark}
\newtheorem{?}[thm]{Problem}
\newtheorem{exa}[thm]{Example}

\theoremstyle{definition}
\newtheorem*{nt*}{Notation}

\newcommand{\im}{{\rm Im}\,}
\newcommand{\re}{{\rm Re}\,}

\newcommand{\inner}[2]{\left\langle {#1}, {#2} \right\rangle}

\newcommand{\calc}{\mathcal C}

\newcommand {\C} {\mathbb C}

\newcommand {\R} {\mathbb R}
\newcommand {\Z} {\mathbb Z}

\newcommand {\spc} {\mathbb{B}_\ell(\mathbb{C}_{\mathbf{Re}>0})}
\newcommand {\hlp} {\mathbb C_{\mathbf{Re}>0}}
\newcommand {\pa} {\partial}
\newcommand {\be} {\beta}

\newcommand {\bfa} {\mathbf a}
\newcommand {\bfb} {\mathbf b}
\newcommand {\bfc} {\mathbf c}
\newcommand {\bfe} {\mathbf E}
\newcommand {\bfh} {\mathbf h}

\newcommand {\calo} {\mathcal O}
\newcommand {\calj} {\mathcal J}

\newcommand {\lb} {\mathbf{Lb}^2(\mathbb{R}_{\geq 0})}

\newcommand {\lp} {\mathcal L}

\begin{document}

\title[]{Complex symmetric weighted composition operators on Bergman spaces and Lebesgue spaces}

\author{Pham Viet Hai* and Osmar R. Severiano}%
\address[P. V. Hai]{...}
\email{phamviethai86@gmail.com}

\address[O. R. Severiano]{Programa Associado de P\'{o}s Gradua\c{c}\~{a}o em Matem\'{a}tica  Universidade Federal da Para\'{i}­ba/Universidade Federal de Campina Grande, Jo\~{a}o Pessoa, Brazil.}
\email{osmar.rrseveriano@gmail.com}



\begin{abstract}
In the paper, we investigate weighted composition operators on Bergman spaces of a half-plane. We characterize weighted composition operators which are hermitian and those which are complex symmetric with respect to a family of conjugations. As it turns out, weighted composition operators enhanced by a symmetry must be bounded. Hermitian, and unitary weighted composition operators are proven to be complex symmetric with respect to an adapted and highly relevant conjugation. We classify which the linear fractional functions give rise to the complex symmetry of bounded composition operators. We end the paper with a natural link to complex symmetry in Lebesgue space.
\end{abstract}

\maketitle

\section{Introduction}
It was proven long by Banach and Stone \cite{zbMATH00467266} that linear, surjective isometries between spaces of continuous functions on compact Hausdorff spaces are of type of a weighted composition operator. The Banach-Stone theorem has been gradually generalized and extended in many different contexts. Among the existing works, we can mention \cite{zbMATH03158247, zbMATH03213975} which show that the only surjective isometries of Hardy spaces $H^p$ with exponent $p\geq 1$ are precisely of this type, except the case $p=2$. Not only does the theorem give rise to inspiration and motivation to study isometries, but also sets the stage for the study of weighted composition operators. 

The paper works with very general weighted composition operators. Specifically, we develop the research in full generality, meaning that the corresponding operators are not assumed to be bounded. Note that unbounded operators are understood as they are not necessarily bounded. We pause a while to get closer the relevant definitions. For analytic functions $f:U\to\C$ and $g:U\to U$, we define the weighted composition expression $\bfe_{f,g}$ by
\begin{gather*}
    \bfe_{f,g}h=f\cdot h\circ g.
\end{gather*}
If $\calo$ is a Banach space of analytic functions over $U$, then the \emph{maximal weighted composition operator} is defined as
\begin{gather*}
    \text{dom}(W_{f,g,\max})=\{h(\cdot)\in\calo:\bfe_{f,g}h(\cdot)\in\calo\},\\
    W_{f,g,\max}h(\cdot)=\bfe_{f,g}h(\cdot),\quad\text{for $h\in\text{dom}(W_{f,g,\max})$}.
\end{gather*}
We relax the domain assumption to define the \emph{nonmaximal weighted composition operator}
\begin{gather*}
    W_{f,g}\preceq W_{f,g,\max},
\end{gather*}
where the symbol $A\preceq B$ means that $\text{dom}(A)\subseteq \text{dom}(B)$  and $Ax=Bx$ for $x\in\text{dom}(A)$, and the term "nonmaximal" is understood as "not necessarily maximal". The operator $W_{f,g}$ is called \emph{bounded} on $\calo$ if the domain $\text{dom}(W_{f,g})=\calo$ and there exists a constant $L>0$ such that $\|W_{f,g}h\|\leq L\|h\|$ for every $h(\cdot)\in\calo$. To simplify terms, when the weight function $f(\cdot)$ is identically $1$, we call as \emph{composition operator}.

The most of research on composition operators is done on Hardy spaces $H^p$ of analytic functions over the unit disk or unit ball. The books \cite{zbMATH00918588, zbMATH00473375, zbMATH00467266} are excellent references. 
One is the fact that every analytic self-map of the unit disk gives rise to a bounded composition operator on Hardy spaces, which is a well-known consequence of the Littlewood Subordination Principle (see \cite{zbMATH00918588}). In the higher dimensional case, there are many examples (see \cite{zbMATH04000348}) which show that composition operators need not be bounded. The classification of normal composition operators can be carried out easily with the conclusion that the only transformation $g(z)=\bfa z$ induces normality (see \cite{zbMATH00918588}). The situation of weighted composition operators is much more delicate to delineate. In \cite{zbMATH05815825}, Cowen and Eungil gave a characterization of hermitian weighted composition operators. Bourdon and Narayan \cite{zbMATH05687699} characterized unitary weighted composition operators and then used this characterization to describe normal weighted composition operators in the case when the function $g(\cdot)$ fixes a point in the unit disk. This case also was considered in \cite{zbMATH06479106} for the study of cohyponormal weighted composition operators. Later, Le \cite{zbMATH06074411} extended the results obtained in \cite{zbMATH05687699, zbMATH05815825} to higher dimensional Hardy spaces. Although Hardy spaces of the disk and half-plane are isomorphic, their composition operators work very differently; for example, in the half-plane not all composition operators are bounded (see \cite{zbMATH05907249}). 

The situation on Fock spaces is studied relatively completely. Characterizations of boundedness, compactness, isometry, normality were produced in \cite{zbMATH06324457, zbMATH06803895}. In \cite{zbMATH07216720}, the author considered unbounded weighted composition operators on Fock space. Extending the corresponding result of \cite{zbMATH06324457}, the author obtained a characterization of unbounded weighted composition operators which are normal and also those which are cohyponormal.

Recently, researchers are interested in the problem of classifying complex symmetric weighted composition operators. Recall that an unbounded, linear operator $B$ on a complex, separable Hilbert space is called \emph{complex symmetric} if there exists an isometric involution $\calc$ (often called as \emph{conjugation}) such that $B=\calc B^*\calc$ (see \cite{zbMATH02237890, zbMATH05148120, zbMATH06349831}). Garcia and Hammond \cite{zbMATH06454968} and Jung et al. \cite{zbMATH06320823} conducted research independently in Hardy spaces $H^p$. One of their results is to characterize a complex symmetric, bounded weighted composition operator, when conjugation is of the concrete form
\begin{gather}\label{con-J}
    \calj h(z)=\overline{h(\overline{z})}.
\end{gather}
Since then the problem has been studied in various function spaces. We refer the reader to \cite{zbMATH06561994} for Hardy spaces of higher dimensions and to \cite{zbMATH06882577} for Lebesgue spaces of measurable functions. For Fock space, the reader can consult the paper \cite{zbMATH06487337}, in which the author and Khoi found out a three parameter family of canonical isometric involutions, containing $\calj$ as a very particular case. Noor and the second author \cite{zbMATH07186914} studied the problem in Hardy space of the open right half-plane. Later, Han and Wang \cite{zbMATH07334460} characterized composition operators that are complex symmetric with respect to conjugation \eqref{con-J}. We emphasize that the paper \cite{zbMATH07334460} deals only with \emph{unweighted} composition operators and the case of \emph{weighted} composition operators is not considered. Recently, there is the work \cite{zbMATH07190521} done in Newton space. As it turns out, the function $g(\cdot)$ gives rise to a complex symmetric, bounded composition operator on Newton space is quite restrictive; including the identity function or zero function. In contrast to the bounded case, the study of unbounded weighted composition operators is at a rather early stage of development. In \cite{zbMATH07216720}, the first author began to study the symmetric properties of unbounded weighted composition operators, comprising real symmetry, complex symmetry, or normality.

It is the goal of the paper to describe the symmetric properties of weighted composition operators acting on Bergman spaces of the open right half-plane. The rest of the paper is organized as follows. Section \ref{sec-prepare} makes some preparation for the proofs of main results, by giving basic definitions, proving a few observations. Section \ref{sec-her} characterizes weighted composition operators which are hermitian and Sections \ref{sec-Ca-self}-\ref{sec-cs-final} describe those which are complex symmetric. We obtain the interesting fact that a weighted composition operator enhanced by a symmetry must be bounded. In Section \ref{sec-cor}, hermitian, and unitary weighted composition operators are proven to be complex symmetric with respect to an adapted and highly relevant conjugation. In Section \ref{sec-com-op}, we classify which the linear fractional functions give rise to the complex symmetry of bounded composition operators. We end the paper with Section \ref{sec-link-Lebes}, in which a natural link to complex symmetry in Lebesgue spaces is established. Section \ref{sec-link-Lebes} is motivated by the Paley-Wiener theorem, which states that Bergman space of the open right half-plane is isometrically isomorphic under the Laplace transform to Lebesgue space.

\section{Preparation}\label{sec-prepare}
Let $\hlp$ be the open right half-plane $\{z\in\C:\re z>0\}$. Throughout the paper, we always assume that $\ell\in\Z_{\geq 0}$. This assumption is essential, as in general, for any $x,y\in\C$ and $\ell>0$, $(xy)^\ell\ne x^\ell y^\ell$ while the equality holds when $\ell\in\Z_{\geq 0}$. For $\ell\in\Z_{\geq 0}$, the weighted Bergman space $\spc$ consists of those analytic functions $h:\hlp\to\C$ for which
\begin{gather*}\label{10}
\|h\|=\left(\dfrac{1}{\pi}
\int\limits_{-\infty}^\infty
\int\limits_0^\infty x^\ell|h(x+iy)|^2\,dxdy\right)^{1/2}<\infty.
\end{gather*}
For each $z\in \hlp,$ the function
\begin{gather*}
K_z(x)=\dfrac{2^\ell(1+\ell)}{(x+\overline{z})^{\ell+2}},\quad x\in \hlp
\end{gather*}
is called the \emph{reproducing kernel} for $\spc$ at $z.$ The kernel functions satisfy the fundamental property $ h(z)=\inner{h}{K_z}$ for each $h\in \spc.$

The next result will be used frequently, as it shows how the adjoint of weighted composition operators act on kernel functions:
\begin{prop}\label{W*Kz-prop}
	For every $z\in\hlp$, we always have $K_z\in\text{dom}(W_{f,g}^*)$ and moreover
	\begin{gather*}
	W_{f,g}^*K_z=\overline{f(z)}K_{g(z)}.
	\end{gather*}
\end{prop}
\begin{proof}
	The proof makes use of the fundamental property of kernel functions and it is left to the reader.
\end{proof}


We establish a condition for a linear fractional function to be a self-mapping of $\hlp$.
\begin{lem}\label{lem-self-map}
Let $\phi(\cdot)$ be the function given by
\begin{gather}\label{eq-bdd-special}
    \phi(z)=-p-\dfrac{q}{z-u},
\end{gather}
where $p,q,u$ are complex constants. Then $\phi$ is a self-mapping of $\hlp$ if and only if
\begin{gather}\label{cond-self-mapping}
    \begin{cases}
    \text{either $\re p=\im q=0,\re q<0,\re u\leq 0$},\\
    \\
    \text{or $\re p<0\leq-\re u+\dfrac{\re q+|q|}{2\re p}$}.
    \end{cases}
\end{gather}
\end{lem}
\begin{proof}
For any $w\in\C$, we denote $w_1=\re w$ and $w_2=\im w$. Hence
\begin{gather*}
    \phi(z)=-p_1-ip_2-\dfrac{(q_1+iq_2)[z_1-u_1-i(z_2-u_2)]}{(z_1-u_1)^2+(z_2-u_2)^2},
\end{gather*}
from which we obtain
\begin{gather*}
\re\phi(z)
=-p_1-\dfrac{q_1(z_1-u_1)+q_2(z_2-u_2)}{(z_1-u_1)^2+(z_2-u_2)^2}.
\end{gather*}
Thus, $\phi(\cdot)$ is a self-mapping of $\hlp$ if and only if
\begin{gather}\label{ineq-self-map-1}
p_1[(z_1-u_1)^2+(z_2-u_2)^2]+q_1(z_1-u_1)+q_2(z_2-u_2)<0,\quad\forall z_1>0,z_2\in\R.
\end{gather}
This inequality gives $p_1\leq 0.$ Indeed,
\begin{align*}
p_1=\lim_{z_2\longrightarrow \infty}p_1\left[ \frac{(z_1-u_1)^2}{(z_2-u_2)^2}+\frac{(z_1-u_1)^2}{(z_2-u_2)^2}\right]\leq- \lim_{z_2\longrightarrow \infty} \left[ q_1\frac{z_1-u_1}{(z_2-u_2)^2}+q_2\frac{z_2-u_2}{(z_2-u_2)^2}\right] =0
\end{align*}
We now study the cases, $p_1=0$ and $p_1<0,$  separately.

\textbf{Case 1:} $p_1=0.$ Then \eqref{ineq-self-map-1} occurs if and only if $q_2=0\geq u_1$ and $q_1<0$.

Suppose that \eqref{ineq-self-map-1} is holds, then 
\begin{align*}q_1(z_1-u_1)<-q_2(z_2-u_2),\quad \forall z_1>0,z_2\in\R.
\end{align*}
For $z_2=u_2,$ we obtain $q_1(z_1-u_1)<0$ for all $z_1>0,$ which implies  $q_1<0$ and $u_1\leq 0.$ From these conditions, we obtain
\begin{align*}
0=\sup\left\lbrace q_1(z_1-u_1):z_1>0\right\rbrace \leq -q_2(z_2-u_2), \quad z_2\in \R.
\end{align*}
Since $z_2\in \R,$ we can find $z_2, z_2'\in \R$ such that $z_2-u_2>0$ and $z_2'-u_2<0.$ Hence
\begin{align*}
0\leq - \left( \frac{z_2-u_2}{z_2-u_2}\right)  q_2=-q_2 \quad \text{and} \quad
0\geq - \left( \frac{z_2'-u_2}{z_2'-u_2}\right)  q_2=-q_2
\end{align*}
whose solution is $q_2=0.$ The converse follows from a quick computation.

\textbf{Case 2:} $p_1<0.$ For this case, \eqref{ineq-self-map-1} turn into
\begin{gather}\label{15}
(z_1-u_1)^2+\dfrac{q_1}{p_1}(z_1-u_1)+(z_2-u_2)^2+\dfrac{q_2}{p_1}(z_2-u_2)>0,\quad\forall z_1>0,z_2\in\R.
\end{gather}
Since the minimum of the function $z\in \R\mapsto z^2+\frac{q_2}{p_1}z$ is $-\frac{q_2^2}{4p_1^2},$ it follows from \eqref{15} that
\begin{gather*}
(z_1-u_1)^2+\dfrac{q_1}{p_1}(z_1-u_1)\geq\dfrac{q_2^2}{4p_1^2},\quad\forall z_1>0,
\end{gather*}
and hence 
\begin{gather*}
    \left|z_1-u_1+\dfrac{q_1}{2p_1}\right|\geq-\dfrac{\sqrt{q_1^2+q_2^2}}{2p_1},\quad\forall z_1>0.
\end{gather*}
Then, the inequality above occurs if and only if the following inequality holds
\begin{gather*}
u_1\leq\dfrac{q_1+\sqrt{q_1^2+q_2^2}}{2p_1}.
\end{gather*}
The proof is complete.
\end{proof}
Before proceeding further, we need the auxiliary result, which says that maximal weighted composition operators are always closed.
\begin{prop}\label{prop-closed-op}
	If $f:\hlp\to\C$ and $g:\hlp\to\hlp$ are analytic functions, then the maximal operator $W_{f,g,\max}$ is closed.
\end{prop}

Now we focus on boundedness of weighted composition operators induced by linear fractional functions.
\begin{lem}\label{lem-bdd-psi-phi}
Let $\ell\in\Z_{\geq 0}$. If $\phi(\cdot)$ is the self-mapping of $\hlp$ given by \eqref{eq-bdd-special}, where coefficients verify \eqref{cond-self-mapping} and $\psi:\hlp\to\C$ is the analytic function given by
\begin{gather*}
    \psi(z)=\dfrac{1}{(z-u)^{\ell+2}}.
\end{gather*}
Then the operator $W_{\psi,\phi,\max}$ is bounded on $\spc$.
\end{lem}
\begin{proof}
By Proposition \ref{prop-closed-op}, it is enough to show that
\begin{gather}\label{eq-dom=spc}
    \text{dom}(W_{\psi,\phi,\max})=\spc.
\end{gather}

Let $h(\cdot)\in\spc$. Consider two cases as follows.

\textbf{Case 1:} $q=0$. This case gives
\begin{gather*}
    \bfe_{\psi,\phi}h(z)=\dfrac{1}{(z-u)^{\ell+2}}h(-p)=\dfrac{h(-p)}{2^\ell(1+\ell)}K_{-\overline{u}}(z)\\
    =\dfrac{1}{2^\ell(1+\ell)}\inner{h}{K_{-p}}K_{-\overline{u}}(z),
\end{gather*}
and hence we can estimate
\begin{gather*}
    \int\limits_{-\infty}^\infty
    \int\limits_0^\infty z_1^\ell|\bfe_{\psi,\phi}h(z_1+iz_2)|^2\,dz_1dz_2
    =\dfrac{1}{2^\ell(1+\ell)}|\inner{h}{K_{-p}}|^2\pi\|K_{-\overline{u}}\|^2\\
    \leq\pi\left(\|K_{-p}\|\cdot\|K_{-\overline{u}}\|\right)^2\|h\|^2.
\end{gather*}
The inequality above shows $h(\cdot)\in\text{dom}(W_{\psi,\phi,\max})$, which provides, as $h(\cdot)$ is arbtrary, that \eqref{eq-dom=spc} holds.

\textbf{Case 2:} $q\ne 0$. Consider
\begin{gather*}
    \int\limits_{-\infty}^\infty
    \int\limits_0^\infty z_1^\ell|\bfe_{\psi,\phi}h(z_1+iz_2)|^2\,dz_1dz_2\\
    =\int\limits_{-\infty}^\infty
    \int\limits_0^\infty\dfrac{z_1^\ell}{((z_1-u_1)^2+(z_2-u_2)^2)^{\ell+2}}
    \Bigg|h\Bigg(-p_1-\dfrac{q_1(z_1-u_1)+q_2(z_2-u_2)}{(z_1-u_1)^2+(z_2-u_2)^2}\\
    +i\left(-p_2-\dfrac{q_2(z_1-u_1)-q_1(z_2-u_2)}{(z_1-u_1)^2+(z_2-u_2)^2}\right)\Bigg)\Bigg|^2\,dz_1dz_2.
\end{gather*}
Let us do the change of variables
\begin{gather*}
    \bfb_1=-p_1-\dfrac{q_1(z_1-u_1)+q_2(z_2-u_2)}{(z_1-u_1)^2+(z_2-u_2)^2},\\
    \bfb_2=-p_2-\dfrac{q_2(z_1-u_1)-q_1(z_2-u_2)}{(z_1-u_1)^2+(z_2-u_2)^2}.
\end{gather*}
Then
\begin{gather*}
    (z_1-u_1)^2+(z_2-u_2)^2=\dfrac{q_1^2+q_2^2}{(\bfb_1+p_1)^2+(\bfb_2+p_2)^2},\\
    z_1=u_1+\dfrac{-q_1(\bfb_1+p_1)-q_2(\bfb_2+p_2)}{(\bfb_1+p_1)^2+(\bfb_2+p_2)^2},\\
    z_2=u_2+\dfrac{-q_2(\bfb_1+p_1)+q_1(\bfb_2+p_2)}{(\bfb_1+p_1)^2+(\bfb_2+p_2)^2},
\end{gather*}
and the Jacobian determinant is
\begin{gather*}
    \mathbf{J}_{z_1,z_2}(\bfb_1,\bfb_2)=
    \begin{vmatrix}
\dfrac{\pa z_1}{\pa\bfb_1} & \dfrac{\pa z_1}{\pa\bfb_2}\\
&\\
\dfrac{\pa z_2}{\pa\bfb_1} & \dfrac{\pa z_2}{\pa\bfb_2}
\end{vmatrix}
=\dfrac{q_1^2+q_2^2}{\left((\bfb_1+p_1)^2+(\bfb_2+p_2)^2\right)^2}.
\end{gather*}
For denoting
\begin{gather*}
    \Omega=\bigg\{\bfb=\bfb_1+i\bfb_2\in\C:u_1[(\bfb_1+p_1)^2+(\bfb_2+p_2)^2]\\
    -q_1(\bfb_1+p_1)-q_2(\bfb_2+p_2)>0\bigg\},
\end{gather*}
we have
\begin{gather*}
    z_1+iz_2\in\hlp\quad\Longleftrightarrow\quad\bfb_1+i\bfb_2\in\Omega.
\end{gather*}
For this change of variables, we continue to establish
\begin{gather*}
    \int\limits_{-\infty}^\infty
    \int\limits_0^\infty z_1^\ell|\bfe_{\psi,\phi}h(z_1+iz_2)|^2\,dz_1dz_2\\
    =\iint\limits_\Omega
    \left[u_1((\bfb_1+p_1)^2+(\bfb_2+p_2)^2)-q_1(\bfb_1+p_1)-q_2(\bfb_2+p_2)\right]^\ell\\
    \times|h(\bfb_1+i\bfb_2)|^2\,d\bfb_1 d\bfb_2.
\end{gather*}
By \eqref{cond-self-mapping}, if $u_1\ne 0$, then we estimate
\begin{align*}
    0&<u_1((\bfb_1+p_1)^2+(\bfb_2+p_2)^2)-q_1(\bfb_1+p_1)-q_2(\bfb_2+p_2)\\
    &=u_1\bfb_1^2+(2u_1p_1-q_1)\bfb_1+u_1p_1^2-p_1q_1-\dfrac{q_2^2}{4u_1}
    +u_1\left(\bfb_2+p_2-\dfrac{q_2}{2u_1}\right)^2\\
    &\leq (2u_1p_1-q_1)\bfb_1.
\end{align*}
Thus, the following is obtained
\begin{gather*}
    \int\limits_{-\infty}^\infty
    \int\limits_0^\infty z_1^\ell|\bfe_{\psi,\phi}h(z_1+iz_2)|^2\,dz_1dz_2\\
    \leq
    \iint\limits_\Omega
    \bfb_1^\ell|h(\bfb_1+i\bfb_2)|^2\,d\bfb_1 d\bfb_2
    \times
    \begin{cases}
        (2u_1p_1-q_1)^\ell,\quad\text{if $u_1\ne 0$},\\
        \\
        (-q_1)^\ell,\quad\text{if $u_1=0$}
    \end{cases}\\
    \leq
    \iint\limits_{\hlp}
    \bfb_1^\ell|h(\bfb_1+i\bfb_2)|^2\,d\bfb_1 d\bfb_2
    \times
    \begin{cases}
        (2u_1p_1-q_1)^\ell,\quad\text{if $u_1\ne 0$},\\
        \\
        (-q_1)^\ell,\quad\text{if $u_1=0$},
    \end{cases}
\end{gather*}
which gives $h(\cdot)\in\text{dom}(W_{\psi,\phi,\max})$, and so \eqref{eq-dom=spc} holds.
\end{proof}

\section{Hermitian property}\label{sec-her}
Recall that an unbounded, linear operator $K$ is called \emph{hermitian} if $K=K^*$. In this section, we investigate how the function-theoretic properties of $f(\cdot)$ and $g(\cdot)$ affect the hermitian property of $W_{f,g}$ and vice versa. As it will be seen in full detail in Theorems \ref{thm-her-1} and \ref{thm-her-2}, the hermitian property of $W_{f,g}$ restricts significantly the possible functions $f(\cdot)$ and $g(\cdot)$. Moreover, the hermitian property ensures the boundedness of $W_{f,g}$.
\begin{lem}\label{lem-her}
Let $\ell\in\Z_{\geq 0}$. Suppose that $f:\hlp\to\C$ and $g:\hlp\to\hlp$ are analytic functions with the property
\begin{gather}\label{cond-her-0}
    \overline{f(z)}K_{g(z)}(w)
    =f(w)K_z(g(w)),\quad\forall z,w\in\hlp.
\end{gather}
Then there are three cases of the functions $f(\cdot)$ and $g(\cdot)$.
\begin{enumerate}
\item The first case is
\begin{gather}\label{form-a-b-del-her-III}
    g(w)=\mu,\quad f(w)=\dfrac{\varepsilon}{(w+\overline{\mu})^{\ell+2}},
\end{gather}
where coefficients satisfy
\begin{gather}\label{cond-a-b-del-her-III}
    \mu\in\hlp,\quad\varepsilon\in\R.
\end{gather}
\item The second case is
\begin{gather}\label{form-a-b-del-her-II}
    g(w)=w+\gamma,\quad f(w)=\lambda,
\end{gather}
where coefficients satisfy
\begin{gather}\label{cond-a-b-del-her-II}
    \gamma\in\R_{\geq 0},\quad\lambda\in\R.
\end{gather}
    \item The third case is
    \begin{gather}\label{form-her}
    g(w)=-\dfrac{\be}{\alpha}-\dfrac{2}{\alpha w-\overline{\be}}
    ,\quad f(w)=\dfrac{\delta}{(\alpha w-\overline{\be})^{\ell+2}},
\end{gather}
where coefficients satisfy
\begin{gather}\label{cond-self-map-0}
\alpha,\delta\in\R\setminus\{0\},\quad\be\in\C,\quad
    \begin{cases}
    \text{either $\re\be<0<\alpha\leq\frac{(\re\be)^2}{2}$},\\
    \\
    \text{or $\alpha<0=\re\be$}.
    \end{cases}
\end{gather}
\end{enumerate}
\end{lem}
\begin{proof}
Condition \eqref{cond-her-0} is equivalent to
\begin{gather}\label{cond-her-1}
    \overline{f(z)}(g(w)+\overline{z})^{\ell+2}
    =f(w)(w+\overline{g(z)})^{\ell+2},\quad\forall z,w\in\hlp,
\end{gather}
which infers that the function $f(\cdot)$ never vanishes. We rewrite \eqref{cond-her-1} in the following form
\begin{gather}\label{eq-02-26-939}
    \dfrac{f(w)}{\overline{f(z)}}
    =\left(\dfrac{g(w)+\overline{z}}{w+\overline{g(z)}}\right)^{\ell+2},\quad\forall z,w\in\hlp.
\end{gather}
Differentiating both sides with respect to the variable $w$ (i.e. taking derivative $\pa_w$ on both sides of \eqref{eq-02-26-939}), we get
\begin{gather}
    \dfrac{f'(w)}{\overline{f(z)}}
    \nonumber=(\ell+2)\left(\dfrac{g(w)+\overline{z}}{w+\overline{g(z)}}\right)^{\ell+1}\\
    \label{eq-02-26-940}\times\dfrac{g'(w)(w+\overline{g(z)})-(g(w)+\overline{z})}{(w+\overline{g(z)})^2},\quad\forall z,w\in\hlp.
\end{gather}
By \eqref{eq-02-26-939} and \eqref{eq-02-26-940}, the following is obtained
\begin{gather}\label{eq-202103051109}
    \dfrac{f'(w)}{f(w)}
    =(\ell+2)\left(\dfrac{g'(w)}{g(w)+\overline{z}}-\dfrac{1}{w+\overline{g(z)}}\right),
\end{gather}
which implies, after taking derivative $\pa_{\overline{z}}$, that
\begin{gather}\label{eq-add-02-25}
    (w+\overline{g(z)})^2g'(w)=(g(w)+\overline{z})^2\overline{g'(z)},\quad\forall z,w\in\hlp.
\end{gather}
Consider two cases as follows.

\textbf{Case 1:} $g''(\cdot)\equiv 0$. Then $g(z)=uz+v$, where $u,v$ are complex constants. Note that
\begin{gather*}
    \text{$g(\cdot)$ is a self-mapping of $\hlp$}\quad\Longleftrightarrow\quad 
    \begin{cases}
    \text{either $u\in\R_{>0},v\in\C_{\mathbf{Re}\geq 0}$},\\
    \\
    \text{or $u=0,v\in\hlp$}.
    \end{cases}
\end{gather*}
If $u=0,v\in\hlp$, then $g(w)=v$ and accordingly \eqref{cond-her-1} is equal to
\begin{gather*}
    \overline{f(z)}(\overline{z}+v)^{\ell+2}
    =f(w)(w+\overline{v})^{\ell+2},\quad\forall z,w\in\hlp.
\end{gather*}
The equality above means that $f(\cdot)(\cdot+\overline{v})^{\ell+2}$ is a real-valued constant function and we get \eqref{form-a-b-del-her-III}-\eqref{cond-a-b-del-her-III}. When $u\in\R_{>0},v\in\C_{\mathbf{Re}\geq 0}$, we write \eqref{eq-add-02-25} in the explicit form
\begin{gather*}
    (w+u\overline{z}+\overline{v})^2u
    =(uw+\overline{z}+v)^2u,\quad\forall z,w\in\hlp.
\end{gather*}
After equating coefficients of $w,\overline{z}$, we get
\begin{gather*}
    u=1,\quad v\in\R_{\geq 0}.
\end{gather*}
It follows from \eqref{cond-her-1} that $f(\cdot)$ is a real-valued constant function and we get \eqref{form-a-b-del-her-II}-\eqref{cond-a-b-del-her-II}.

\textbf{Case 2:} $g''(\cdot)\not\equiv 0$. Let us take derivative $\pa_{\overline{z}}\circ\pa_w$ on both sides of \eqref{eq-add-02-25} to get
\begin{gather*}
    (w+\overline{g(z)})g''(w)\overline{g'(z)}\\
    =(g(w)+\overline{z})g'(w)\overline{g''(z)},\quad\forall z,w\in\hlp.
\end{gather*}
Setting
\begin{gather*}
    h(w)=\dfrac{g''(w)}{g'(w)},
\end{gather*}
the equality above turns into
\begin{gather}\label{cond-her-2}
    h(w)(w+\overline{g(z)})
    =\overline{h(z)}(g(w)+\overline{z}),\quad\forall z,w\in\hlp,
\end{gather}
which implies, after taking $\pa_{\overline{z}}\circ\pa_w$, that
\begin{gather*}
    h'(w)\overline{g'(z)}
    =g'(w)\overline{h'(z)},\quad\forall z,w\in\hlp.
\end{gather*}
Consequently, there exist constants $\alpha\in\R,\be\in\C$, such that
\begin{gather*}
    h'(w)=\alpha g'(w)
    \Longrightarrow h(w)=\alpha g(w)+\be.
\end{gather*}
Through substituting this form of $h(\cdot)$ back into \eqref{cond-her-2}, and then rearranging terms, we get
\begin{gather*}
    \alpha wg(w)+\be w-\overline{\be}g(w)\\
    =\alpha\overline{zg(z)}+\overline{\be z}-\be\overline{g(z)},\quad\forall z,w\in\hlp.
\end{gather*}
Accordingly, there is a constant $\gamma\in\R$ for which
\begin{gather*}
    g(w)=\dfrac{-\be w+\gamma}{\alpha w-\overline{\be}}.
\end{gather*}

- If $\alpha=\be=0$, then $h\equiv 0$ and hence $g''(w)\equiv 0$; but this is impossible.

- If $\alpha=0,\be\ne 0$, then
\begin{gather*}
    g(w)=\dfrac{\be}{\overline{\be}}z-\dfrac{\gamma}{\overline{\be}}.
\end{gather*}
Note that
\begin{gather*}
    \text{$g(\cdot)$ is a self-mapping of $\hlp$}\quad\Longleftrightarrow\quad 
    \dfrac{\be}{\overline{\be}}\in\R_{>0},\dfrac{\gamma}{\overline{\be}}\in\C_{\textbf{Re}\leq 0}\\
    \Longleftrightarrow\be\in\R,\dfrac{\gamma}{\be}\in\C_{\textbf{Re}\leq 0}.
\end{gather*}
By \eqref{eq-202103051109}-\eqref{eq-add-02-25}, we get \eqref{form-a-b-del-her-II}-\eqref{cond-a-b-del-her-II}.

- If $\alpha\ne 0$, then
\begin{gather*}
    g(w)
    =-\dfrac{\be}{\alpha}+\left(\gamma-\dfrac{|\be|^2}{\alpha}\right)\dfrac{1}{\alpha w-\overline{\be}}.
\end{gather*}
For that reason, we have
\begin{gather*}
    -\dfrac{2\alpha}{\alpha w-\overline{\be}}=\dfrac{g''(w)}{g'(w)}=h(w)\\
    =\alpha g(w)+\be=\alpha\left(\gamma-\dfrac{|\be|^2}{\alpha}\right)\dfrac{1}{\alpha w-\overline{\be}},
\end{gather*}
which gives
\begin{gather*}
    -2=\gamma-\dfrac{|\be|^2}{\alpha}.
\end{gather*}
Thus, the function $g(\cdot)$ is of form
\begin{gather*}
    g(w)=-\dfrac{\be}{\alpha}-\dfrac{2}{\alpha w-\overline{\be}},
\end{gather*}
and so, \eqref{cond-her-1} is simplified to
\begin{gather*}
    \dfrac{\overline{f(z)}}{(\alpha w-\overline{\be})^{\ell+2}}
    =\dfrac{f(w)}{(\alpha\overline{z}-\be)^{\ell+2}},\quad\forall z,w\in\hlp.
\end{gather*}
This is equivalent to the fact that $(\alpha\cdot-\overline{\be})^{\ell+2}f(\cdot)$ is a real-valued constant function and we obtain \eqref{form-her}. Note that \eqref{cond-self-map-0} follows directly from Lemma \ref{lem-self-map}.
\end{proof}

\begin{rem}\label{rem-202103051127}
If the functions $f(\cdot)$ and $g(\cdot)$ are as in \eqref{form-a-b-del-her-II}-\eqref{cond-a-b-del-her-II}, then by \cite{zbMATH05907249} the composition operator $C_{g,\max}$ is bounded on $\spc$ and so the weighted composition operator $W_{f,g,\max}$ is automatically bounded on $\spc$.
\end{rem}

\begin{thm}\label{thm-her-1}
Let $\ell\in\Z_{\geq 0}$. Let $f:\hlp\to\C$ and $ g:\hlp\to\hlp$ be analytic functions. Then the operator $W_{f,g,\max}$ is hermitian on $\spc$ if and only if \eqref{form-a-b-del-her-III}-\eqref{cond-self-map-0} hold. In this case, the operator $W_{f,g,\max}$ is bounded.
\end{thm}
\begin{proof}
Suppose that the operator $W_{f,g,\max}$ is hermitian, which means $W_{f,g,\max}^*=W_{f,g,\max}$. Then
\begin{gather}\label{eq-W*=W}
    W_{f,g,\max}^*K_z=W_{f,g,\max}K_z,\quad\forall z\in\hlp.
\end{gather}
By Proposition \ref{W*Kz-prop}, equation \eqref{eq-W*=W} turns into \eqref{cond-her-0} and so we make use of Lemma \ref{lem-her} to get \eqref{form-a-b-del-her-III}-\eqref{cond-self-map-0}.

Conversely, take $f(\cdot)$ and $g(\cdot)$ as in the statement of the theorem. Note that by Lemma \ref{lem-bdd-psi-phi} and Remark \ref{rem-202103051127} the operator $W_{f,g,\max}$ is bounded. A direct computation gives \eqref{eq-W*=W}, and hence it is hermitian on the whole space $\spc$. 
\end{proof}

\begin{thm}\label{thm-her-2}
Let $\ell\in\Z_{\geq 0}$. Let $f:\hlp\to\C$ and $g:\hlp\to\hlp$ be analytic functions. Then the operator $W_{f,g}$ is hermitian on $\spc$ if and only if it verifies two conditions.
\begin{enumerate}
    \item The operator $W_{f,g}$ is maximal; that is $W_{f,g}=W_{f,g,\max}$.
    \item \eqref{form-a-b-del-her-III}-\eqref{cond-self-map-0} hold. 
\end{enumerate}
In this case, the operator $W_{f,g}$ is bounded.
\end{thm}
\begin{proof}
The implication $\Longleftarrow$ is obtained from Theorem \ref{thm-her-1}. Let us prove the implication $\Longrightarrow$. First, we show that the operator $W_{f,g,\max}$ is hermitian. Indeed, since $W_{f,g}\preceq W_{f,g,\max}$, by \cite[Proposition 1.6]{KS}, we have
$$
W_{f,g,\max}^*\preceq W_{f,g}^*=W_{f,g}\preceq W_{f,g,\max}.
$$
Proposition \ref{W*Kz-prop} shows that kernel functions always belong to the domain $\text{dom}(W_{f,g,\max}^*)$, and so,
$$
W_{f,g,\max}^*K_z(u)=W_{f,g,\max}K_z(u),\quad\forall z,u\in\hlp.
$$
By Lemma \ref{lem-her}, conditions \eqref{form-a-b-del-her-III}-\eqref{cond-self-map-0} hold, and hence, by Theorem \ref{thm-her-1}, the operator $W_{f,g,\max}$ is hermitian. For that reason, item (1) follows from the following inclusions
$$
W_{f,g}\preceq W_{f,g,\max}=W_{f,g,\max}^*\preceq W_{f,g}^*=W_{f,g}.
$$
\end{proof}

\section{Unitary property}\label{sec-uni-op}
The lemma below is instrumental in the characterization of all unitary weighted composition operators on $\spc$. It allows us to compute explicitly the symbols whose weighted composition operator satisfies $W_{f,g,\max}W_{f,g,\max}^*=I$.
\begin{lem}\label{lem-uni}
Let $\ell\in\Z_{\geq 0}$. Suppose that $f:\hlp\to\C$ and $g:\hlp\to\hlp$ are analytic functions with the property
\begin{gather}\label{as-uni}
    \overline{f(z)}f(x)K_{g(z)}(g(x))
    =K_z(x),\quad\forall x,z\in\hlp.
\end{gather}
Then these functions are of the following forms.
\begin{enumerate}
    \item Either
    \begin{gather}\label{form-uni-2}
    g(z)=|C|^{2/(\ell+2)}z+i\delta,\quad f(z)=C,    
    \end{gather}
    where
    \begin{gather}\label{cond-uni-2}
    C\in\C\setminus\{0\},\quad\delta\in\R,
\end{gather}
    \item or
    \begin{gather}\label{form-uni}
    g(z)=\dfrac{|\beta|^{2/(\ell+2)}}{z-i\alpha}+i\theta,\quad f(z)=\dfrac{\beta}{(z-i\alpha)^{\ell+2}},
    \end{gather}
where
\begin{gather}\label{cond-uni}
    \beta\in\C\setminus\{0\},\quad\alpha,\theta\in\R.
\end{gather}
\end{enumerate}
\end{lem}
\begin{proof}
It can be seen from \eqref{as-uni} that the symbol $f(\cdot)$ never vanishes and moreover the following is obtained
\begin{gather}\label{psiz-psix}
    \overline{f(z)}f(x)=\left(\dfrac{g(x)+\overline{g(z)}}{x+\overline{z}}\right)^{\ell+2}.
\end{gather}
There are two cases of the symbol $f(\cdot)$.

- If $f(\cdot)\equiv C$ is a constant function, then equation \eqref{psiz-psix} is simplified to
\begin{gather}\label{eq-xi-const-0}
    (g(x)+\overline{g(z)})^{\ell+2}
    =|C|^2(x+\overline{z})^{\ell+2}.
\end{gather}
Let us continue to take derivative $\partial_x$ and then
\begin{gather}\label{eq-xi-const-1}
    (g(x)+\overline{g(z)})^{\ell+1}g'(x)
    =|C|^2(x+\overline{z})^{\ell+1}.
\end{gather}
It follows from equations \eqref{eq-xi-const-0}-\eqref{eq-xi-const-1}, that
\begin{gather*}
    \dfrac{(g(x)+\overline{g(z)})^{\ell+1}g'(x)}
    {(g(x)+\overline{g(z)})^{\ell+2}}
    =\dfrac{|C|^2(x+\overline{z})^{\ell+1}}{|C|^2(x+\overline{z})^{\ell+2}}
\end{gather*}
or equivalently to saying that
\begin{gather}\label{eq-xi-const-2}
    g(x)+\overline{g(z)}
    =(x+\overline{z})g'(x).
\end{gather}
Taking derivative $\partial_{\overline{z}}$, the above equation becomes $\overline{g'(z)}
    =g'(x)$; meaning that $g(x)=\lambda x+\mu$, where $\lambda\in\R$ and $\mu\in\C$. This form of $g(\cdot)$ and equation \eqref{eq-xi-const-2} yield $\re\mu=0$ and so \eqref{form-uni-2}-\eqref{cond-uni-2} hold.

- If $f(\cdot)$ is not a constant function, then we take derivative $\partial_x$ on both sides of \eqref{psiz-psix} and then get
\begin{gather}
    \nonumber\overline{f(z)}f'(x)
    =(\ell+2)\left(\dfrac{g(x)+\overline{g(z)}}{x+\overline{z}}\right)^{\ell+1}\\
    \label{psiz-psix'}\times\dfrac{(x+\overline{z})g'(x)-g(x)-\overline{g(z)}}{(x+\overline{z})^2}.
\end{gather}
By \eqref{psiz-psix} and \eqref{psiz-psix'}, the following is obtained
\begin{gather*}
    (\ell+2)
    \left(\dfrac{g'(x)}{g(x)+\overline{g(z)}}-\dfrac{1}{x+\overline{z}}\right)
    =\dfrac{\overline{f(z)}f'(x)}{\overline{f(z)}f(x)}
    =\dfrac{f'(x)}{f(x)}.
\end{gather*}
Taking derivative $\partial_{\overline{z}}$, we obtain
\begin{gather*}
    -\dfrac{g'(x)\overline{g'(z)}}{(g(x)+\overline{g(z)})^2}
    +\dfrac{1}{(x+\overline{z})^2}=0
\end{gather*}
or equivalently to saying that
\begin{gather}\label{eq-varphi-uni}
    (g(x)+\overline{g(z)})^2=(x+\overline{z})^2
    g'(x)\overline{g'(z)}.
\end{gather}
Taking derivative $\partial_x$, we get
\begin{gather*}
    2(g(x)+\overline{g(z)})g'(x)\\
    =2(x+\overline{z})g'(x)\overline{g'(z)}
    +(x+\overline{z})^2g''(x)\overline{g'(z)},
\end{gather*}
which implies, after taking $\partial_{\overline{z}}$, that
\begin{gather*}
    (x+\overline{z})g''(x)\overline{g''(z)}
    +2g'(x)\overline{g''(z)}
    +2g''(x)\overline{g'(z)}
    =0.
\end{gather*}
Setting
\begin{gather}\label{eq-h}
    h(x)=\dfrac{g'(x)}{g''(x)},
\end{gather}
the equation above reads as
\begin{gather}\label{eq-sub-h=}
    (x+\overline{z})\overline{g''(z)}
    +2h(x)\overline{g''(z)}
    +2\overline{g'(z)}
    =0.
\end{gather}
Taking derivative $\partial_x$, the following is obtained
\begin{gather*}
    \overline{g''(z)}+2h'(x)\overline{g''(z)}=0,
\end{gather*}
and so $h'(x)=-1/2$. For that reason, it can be seen that $h(x)=-x/2+\kappa$, where $\kappa$ is some complex constant. Consequently, taking into account form \eqref{eq-h} of $h(\cdot)$, we have
\begin{gather}\label{eq-C-1}
    g'(x)=\left(-\dfrac{x}{2}+\kappa\right)g''(x).
\end{gather}
Substituting the expression $h(x)=-x/2+\kappa$ back into \eqref{eq-sub-h=}, the below equation holds
\begin{gather}\label{eq-C-2}
    (z+2\overline{\kappa})g''(z)+2g'(z)=0.
\end{gather}
Note that equations \eqref{eq-C-1}-\eqref{eq-C-2} gives
\begin{gather*}
    \kappa\in i\R.
\end{gather*}
Setting $D(z)=(z-2\kappa)^2g'(z)$, then
\begin{gather*}
    g''(z)
    =\dfrac{D'(z)}{(z-2\kappa)^2}-\dfrac{2D(z)}{(z-2\kappa)^3}.
\end{gather*}
Equation \eqref{eq-C-1} is reduced to the following
\begin{gather*}
    D'(z)=0,\quad\forall z\in\hlp,
\end{gather*}
which means that $D(z)\equiv D$ is a constant function. Thus, we have
\begin{gather*}
    g'(z)=\dfrac{D}{(z-2\kappa)^2},
\end{gather*}
and so the symbol $g$ can be expressed as
\begin{gather*}
    g(x)=\dfrac{E}{x-F}+G,\quad\text{where $E=-D,F=2\kappa$ and $G=g(0)+\dfrac{E}{F}$}.
\end{gather*}
Substituting this expression of $g(\cdot)$ back into \eqref{eq-varphi-uni} with the note that $F\in i\R$, we get
\begin{gather*}
    \left(\dfrac{E}{x-F}+\dfrac{\overline{E}}{\overline{z}+F}+2\re G\right)^2
    =\dfrac{|E|^2(x+\overline{z})^2}{(x-F)^2(\overline{z}+F)^2}\\
    =|E|^2\left(\dfrac{1}{x-F}+\dfrac{1}{\overline{z}+F}\right)^2.
\end{gather*}
The last equality occurs if and only if $\re G=0$ and $E^2=|E|^2$; meaning that
\begin{gather*}
    g(z)=\dfrac{E}{z-i\alpha}+i\theta,\quad\text{where $\alpha,\theta\in\R$}.
\end{gather*}
For this $g(\cdot)$, equation \eqref{psiz-psix} gives $f(\cdot)$ as in \eqref{form-uni} with condition \eqref{cond-uni}.
\end{proof}

We now state and prove a characterization of unitary weighted composition operators on $\spc$. It turns out that the unitary property significantly restricts the possible symbols for weighted composition operators.
\begin{thm}\label{thm-uni}
Let $\ell\in\Z_{\geq 0}$. Let $g$ be an analytic self-map of $\hlp$ and $f$ be an analytic function on $\hlp$. Then the operator $W_{f,g,\max}$ is unitary on $\spc$ if and only if its symbols are of forms either \eqref{form-uni-2}-\eqref{cond-uni-2} or \eqref{form-uni}-\eqref{cond-uni}.
\end{thm}
\begin{proof}
The necessary condition can be obtained from Lemma \ref{lem-uni}. Let us continue to prove the sufficient condition. Suppose that the symbols $\xi(\cdot)$ and $\phi(\cdot)$ take forms as in the statement of the theorem. By Lemma \ref{lem-bdd-psi-phi}, the operator $W_{f,g,\max}$ is bounded. A simple computation gives
\begin{gather*}
    W_{f,g,\max}W_{f,g,\max}^*K_z=W_{f,g,\max}^*W_{f,g,\max}K_z=K_z,\quad\forall z\in\hlp,
\end{gather*}
and so the operator $W_{f,g,\max}$ is unitary on $\spc$.
\end{proof}

\section{$\calc_{\bfa}$-selfadjoint property}\label{sec-Ca-self}
Let us define the operator $\calc_{\bfa}$ by setting
\begin{gather}
    \calc_{\bfa} f(z)=\overline{f(\overline{z}+i\bfa)},\quad\bfa\in\R.
\end{gather}

\begin{lem}\label{lem-conj-Ca}
For every $\bfa\in\R$, the operator $\calc_{\bfa}$ is a conjugation on $\spc$ and it acts on kernel functions by the following rule
\begin{gather}\label{eq-CaKz=}
    \calc_{\bfa}K_z=K_{\overline{z}+i\bfa},\quad\forall z\in\hlp.
\end{gather}
\end{lem}

In the section, we describe the weighted composition operators which are complex symmetric with respect to conjugation $\calc_{\bfa}$ (simply: \emph{$\calc_{\bfa}$-selfadjoint}).
\begin{lem}\label{lem-cs-Ca}
Let $\ell\in\Z_{\geq 0}$ and $\bfa\in\R$. Suppose that $f:\hlp\to\C$ and $g:\hlp\to\hlp$ are analytic functions with the property
\begin{gather}\label{cond-cs-Ca-ok}
    f(\overline{z}+i\bfa)K_{\overline{g(\overline{z}+i\bfa)}+i\bfa}(w)
    =f(w)K_z(g(w)),\quad\forall z,w\in\hlp.
\end{gather}
Then there are three cases of the functions $f(\cdot)$ and $g(\cdot)$.
\begin{enumerate}
    \item The first case is
\begin{gather}\label{form-cs-Ca-1-I}
    g(w)=\mu,\quad f(w)=\dfrac{\delta}{(w+\mu-i\bfa)^{\ell+2}},
\end{gather}
where coefficients satisfy
\begin{gather}\label{cond-cs-Ca-I}
    \mu\in\hlp,\quad\delta\in\C.
\end{gather}
    \item The second case is
\begin{gather}\label{form-cs-Ca-1-II}
    g(w)=w+\gamma,\quad f(w)=\lambda,
\end{gather}
where coefficients satisfy
\begin{gather}\label{cond-cs-Ca-II}
    \gamma\in\C_{\mathbf{Re}\geq 0},\quad\lambda\in\C.
\end{gather}
    \item The third case is
\begin{gather}\label{form-cs-Ca-1}
    g(w)=-\dfrac{\be}{\alpha}-\dfrac{2}{\alpha(w-i\bfa)-\be},
    \quad f(w)=\dfrac{\delta}{(\alpha(w-i\bfa)-\be)^{\ell+2}},
\end{gather}
where coefficients satisfy
\begin{gather}\label{cond-self-map-cs-Ca}
\alpha\in\C\setminus\{0\},\quad\be,\delta\in\C,\quad
    \begin{cases}
    \text{either $\re\dfrac{\be}{\alpha}=\im\dfrac{1}{\alpha}=0,\re\dfrac{1}{\alpha}<0$},\\
    \\
    \text{or $\re\dfrac{\be}{\alpha}<0,\left(\re\dfrac{\be}{\alpha}\right)^2\geq\re\dfrac{1}{\alpha}+\dfrac{1}{|\alpha|}$}.
    \end{cases}
\end{gather}    
\end{enumerate}
\end{lem}
\begin{proof}
Equality \eqref{cond-cs-Ca-ok} can be rewritten as
\begin{gather}\label{cond-cs-Ca-2}
    f(z+i\bfa)[g(w)+z]^{\ell+2}
    =f(w)[w+g(z+i\bfa)-i\bfa]^{\ell+2},\quad\forall z,w\in\hlp,
\end{gather}
which implies that the function $f(\cdot)$ never vanishes in $\hlp$ and moreover
\begin{gather}\label{eq-02-26-1347}
    \dfrac{f(w)}{f(z+i\bfa)}
    =\left(\dfrac{g(w)+z}{w+g(z+i\bfa)-i\bfa}\right)^{\ell+2},\quad\forall z,w\in\hlp.
\end{gather}
Let us continue to take derivative $\pa_w$ and get
\begin{gather}
    \dfrac{f'(w)}{f(z+i\bfa)}
    \nonumber=(\ell+2)\left(\dfrac{g(w)+z}{w+g(z+i\bfa)-i\bfa}\right)^{\ell+1}\\
    \label{eq-02-26-1350}\times\dfrac{g'(w)(w+g(z+i\bfa)-i\bfa)-(g(w)+z)}{(w+g(z+i\bfa)-i\bfa)^2},\quad\forall z,w\in\hlp.
\end{gather}
It follows from \eqref{eq-02-26-1347} and \eqref{eq-02-26-1350} that
\begin{gather*}
    \dfrac{f'(w)}{f(w)}
    =(\ell+2)\left(\dfrac{g'(w)}{g(w)+z}-\dfrac{1}{w+g(z+i\bfa)-i\bfa}\right),\quad\forall z,w\in\hlp.
\end{gather*}
After differentiating with respect to the variable $z$, the following is obtained
\begin{gather}\label{eq-add-02-25-15h}
    (w+g(z+i\bfa)-i\bfa)^2g'(w)
    =(g(w)+z)^2g'(z+i\bfa),\quad\forall z,w\in\hlp.
\end{gather}
Consider two cases as follows.

\textbf{Case 1:} $g''(\cdot)\equiv 0$. Then $g(z)=uz+v$, where $u,v$ are complex constants. Note that
\begin{gather*}
    \text{$g(\cdot)$ is a self-mapping of $\hlp$}\quad\Longleftrightarrow\quad 
    \begin{cases}
    \text{either $u\in\R_{>0},v\in\C_{\mathbf{Re}\geq 0}$},\\
    \\
    \text{or $u=0,v\in\hlp$}.
    \end{cases}
\end{gather*}
If $u=0,v\in\hlp$, then $g(w)=v$ and so \eqref{cond-cs-Ca-2} turns into
\begin{gather*}
    f(z+i\bfa)(z+v)^{\ell+2}
    =f(w)(w+v-i\bfa)^{\ell+2},\quad\forall z,w\in\hlp.
\end{gather*}
This means that $f(\cdot)(\cdot+v-i\bfa)^{\ell+2}$ is a complex-valued constant function and we get \eqref{form-cs-Ca-1-I}-\eqref{cond-cs-Ca-I}. When $u\in\R_{>0},v\in\C_{\mathbf{Re}\geq 0}$, equation \eqref{eq-add-02-25-15h} becomes
\begin{gather*}
    (w+u(z+i\bfa)+v-i\bfa)^2u
    =(uw+v+z)^2u,\quad\forall z,w\in\hlp.
\end{gather*}
Equating coefficients of $w,z$, we obtain $u=1$ and then $g(w)=w+v$. Setting $g(w)=w+v$ in \eqref{cond-cs-Ca-2}, we find that $f(\cdot)$ is a complex-valued constant function and \eqref{form-cs-Ca-1-II}-\eqref{cond-cs-Ca-II} hold.

\textbf{Case 2:} $g''(\cdot)\not\equiv 0$. In this case, taking derivative $\pa_z\circ\pa_w$ on both sides of \eqref{eq-add-02-25-15h} we observe
\begin{gather*}
    g''(w)g'(z+i\bfa)(w+g(z+i\bfa)-i\bfa)\\
    =g'(w)g''(z+i\bfa)(g(w)+z),\quad\forall z,w\in\hlp.
\end{gather*}
Setting
\begin{gather*}
    h(w)=\dfrac{g''(w)}{g'(w)},
\end{gather*}
the following is obtained
\begin{gather}\label{cond-cs-Ca-3}
    h(w)(w+g(z+i\bfa)-i\bfa)
    =h(z+i\bfa)(g(w)+z),\quad\forall z,w\in\hlp.
\end{gather}
We go forward taking derivative $\pa_z\circ\pa_w$
\begin{gather*}
    h'(w)g'(z+i\bfa)=h'(z+i\bfa)g'(w),\quad\forall z,w\in\hlp.
\end{gather*}
Hence, there exist constants $\alpha,\be\in\C$ such that
\begin{gather*}
    h'(w)=\alpha g'(w)\Longrightarrow h(w)=\alpha g(w)+\be.
\end{gather*}
For such $h(\cdot),g(\cdot)$, equation \eqref{cond-cs-Ca-3} is simplified to the following
\begin{gather*}
    \alpha g(w)(w-i\bfa)+\be(w-i\bfa)+\be g(z+i\bfa)\\
    =\alpha zg(z+i\bfa)+\be g(w)+\be z,\quad\forall z,w\in\hlp
\end{gather*}
or equivalently to saying that
\begin{gather*}
    \alpha g(w)(w-i\bfa)+\be(w-i\bfa)-\be g(w)\\
    =\alpha g(z+i\bfa)z+\be z-\be g(z+i\bfa),\quad\forall z,w\in\hlp.
\end{gather*}
Accordingly, there exists a constant $\gamma\in\C$ such that
\begin{gather*}
    g(w)=\dfrac{\gamma-\be(w-i\bfa)}{\alpha(w-i\bfa)-\be}.
\end{gather*}

- If $\alpha=\be=0$, then $h(\cdot)\equiv 0$ and so $g''(\cdot)\equiv 0$; but this is simpossible.

- If $\alpha=0,\be\ne 0$, then
\begin{gather*}
    g(w)=w-i\bfa-\dfrac{\gamma}{\be};
\end{gather*}
but this is impossible as $g''(\cdot)\not\equiv 0$.

- If $\alpha\ne 0$, then
\begin{gather*}
    g(w)=-\dfrac{\be}{\alpha}
    +\left(\gamma-\dfrac{\be^2}{\alpha}\right)\dfrac{1}{\alpha(w-i\bfa)-\be}.
\end{gather*}
A direct computation shows
\begin{gather*}
    -\dfrac{2\alpha}{\alpha(w-i\bfa)-\be}=\dfrac{g''(w)}{g'(w)}
    =h(w)=\alpha g(w)+\be=\dfrac{\alpha\gamma-\be^2}{\alpha(w-i\bfa)-\be},
\end{gather*}
which gives
\begin{gather*}
    -2\alpha=\alpha\gamma-\be^2\quad\Longleftrightarrow\quad -2=\gamma-\dfrac{\be^2}{\alpha}.
\end{gather*}
For that reason, we get
\begin{gather*}
    g(w)=-\dfrac{\be}{\alpha}
    -\dfrac{2}{\alpha(w-i\bfa)-\be},
\end{gather*}
and hence \eqref{cond-cs-Ca-2} becomes
\begin{gather*}
    \dfrac{1}{(\alpha w-\be-i\alpha\bfa)^{\ell+2}}f(z+i\bfa)
    =\dfrac{1}{(\alpha z-\be)^{\ell+2}}f(w),\quad\forall z,w\in\hlp.
\end{gather*}
This means that $f(\cdot)(\alpha\cdot-\be-i\alpha\bfa)^{\ell+2}$ is a complex-valued constant function; namely, this case gives \eqref{form-cs-Ca-1}. Note that \eqref{cond-self-map-cs-Ca} follows directly from Lemma \ref{lem-bdd-psi-phi}.
\end{proof}

\begin{thm}\label{thm-Ca-1}
Let $\ell\in\Z_{\geq 0},\bfa\in\R$. Let $f:\hlp\to\C$ and $g:\hlp\to\hlp$ be analytic functions. Then the operator $W_{f,g,\max}$ is $\calc_\bfa$-selfadjoint on $\spc$ if and only if \eqref{form-cs-Ca-1-I}-\eqref{cond-self-map-cs-Ca} hold. In this case, the operator $W_{f,g,\max}$ is bounded.
\end{thm}
\begin{proof}
Suppose that the operator $W_{f,g,\max}$ is $\calc_\bfa$-selfadjoint, which means $\calc_\bfa W_{f,g,\max}^*\calc_\bfa=W_{f,g,\max}$. In particular, we have
\begin{gather}\label{eq-CaW*Ca=W}
    \calc_\bfa W_{f,g,\max}^*\calc_\bfa K_z=W_{f,g,\max}K_z,\quad\forall z\in\hlp.
\end{gather}
By Proposition \ref{W*Kz-prop} and Lemma \ref{lem-conj-Ca}, equation \eqref{eq-CaW*Ca=W} turns into \eqref{cond-cs-Ca-ok}. We make use of Lemma \ref{lem-cs-Ca} to get \eqref{form-cs-Ca-1-I}-\eqref{cond-self-map-cs-Ca}.

Conversely, take $f(\cdot)$ and $g(\cdot)$ as in the statement of the theorem. A direct computation gives \eqref{eq-CaW*Ca=W}, and hence by Lemma \ref{lem-bdd-psi-phi} and Remark \ref{rem-202103051127} the operator $W_{f,g,\max}$ is $\calc_\bfa$-selfadjoint on the whole space $\spc$. 
\end{proof}

\begin{thm}\label{thm-Ca-2}
Let $\ell\in\Z_{\geq 0},\bfa\in\R$. Let $f:\hlp\to\C$ and $g:\hlp\to\hlp$ be analytic functions. Then the operator $W_{f,g}$ is $\calc_\bfa$-selfadjoint on $\spc$ if and only if it verifies two conditions.
\begin{enumerate}
    \item The operator $W_{f,g}$ is maximal; that is $W_{f,g}=W_{f,g,\max}$.
    \item \eqref{form-cs-Ca-1-I}-\eqref{cond-self-map-cs-Ca} hold. 
\end{enumerate}
In this case, the operator $W_{f,g}$ is bounded.
\end{thm}
\begin{proof}
The implication $\Longleftarrow$ is obtained from Theorem \ref{thm-her-1}. Let us prove the implication $\Longrightarrow$. First, we show that the operator $W_{f,g,\max}$ is hermitian. Indeed, since $W_{f,g}\preceq W_{f,g,\max}$, by \cite[Proposition 1.6]{KS}, we have
$$
W_{f,g,\max}^*\preceq W_{f,g}^*=\calc_\bfa W_{f,g}\calc_\bfa\preceq \calc_\bfa W_{f,g,\max}\calc_\bfa.
$$
Proposition \ref{W*Kz-prop} shows that kernel functions always belong to the domain $\text{dom}(W_{f,g,\max}^*)$, and so,
$$
\calc_\bfa W_{f,g,\max}^*\calc_\bfa K_z(u)=W_{f,g,\max}K_z(u),\quad\forall z,u\in\hlp.
$$
By Lemma \ref{lem-cs-Ca}, conditions \eqref{form-cs-Ca-1-I}-\eqref{cond-self-map-cs-Ca} hold, and hence, by Theorem \ref{thm-Ca-1}, the operator $W_{f,g,\max}$ is $\calc_\bfa$-selfadjoint. For that reason, item (1) follows from the following inclusions
$$
W_{f,g}\preceq W_{f,g,\max}=\calc_\bfa W_{f,g,\max}^*\calc_\bfa\preceq \calc_\bfa W_{f,g}^*\calc_\bfa=W_{f,g}.
$$
\end{proof}

\section{$\calc_\star$-selfadjoint property}\label{sec-cs-final}
In Section \ref{sec-Ca-self}, we obtain the interesting fact that the function $g(w)=w+\gamma$ can induce a complex symmetric weighted composition operator on $\spc$. Naturally, we are in question about the case when $g(w)=\lambda w+\gamma$. As it turns out, that case also gives rise to the complex symmetry, but conjugations differ from those used in Section \ref{sec-Ca-self}. Our conjugations are constructed as follows.

Let us define the anti-linear operator $\calc_\star:\spc\to\spc$ by setting
\begin{gather}
    \calc_\star f(z)=\dfrac{1}{z^{\ell+2}}\overline{f\left(\dfrac{1}{\overline{z}}\right)}
\end{gather}

\begin{lem}\label{lem-conj-Cstar}
Let $\ell\in\Z_{\geq 0}$. The operator $\calc_\star$ is a conjugation on $\spc$ and it acts on kernel functions by the following rule
\begin{gather}\label{eq-CstarKz=}
    \calc_\star K_z=\dfrac{1}{z^{\ell+2}} K_{1/\overline{z}}, \quad\forall z\in\hlp.
\end{gather}
\end{lem}
\begin{proof}
    The proof is left to the reader as it is a direct computation.
\end{proof}

\begin{lem}\label{lem-Cstar}
Let $\ell\in\Z_{\geq 0}$. Suppose that $f:\hlp\to\C$ and $g:\hlp\to\hlp$ are analytic functions with the property
\begin{gather}\label{cond-Cstar-sym}
    \dfrac{1}{\overline{z}^{\ell+2}}f(1/\overline{z})
    \dfrac{1}{g(1/\overline{z})^{\ell+2}}K_{1/\overline{g(1/\overline{z})}}(w)
    =f(w)K_z(g(w)),\quad\forall z,w\in\hlp.
\end{gather}
Then there are three cases of the functions $f(\cdot)$ and $g(\cdot)$.
\begin{enumerate}
    \item The first case is
    \begin{gather}\label{form-Cstar-I}
        g(w)=\alpha,\quad f(w)=\dfrac{\be}{(1+\alpha z)^{\ell+2}},
    \end{gather}
    where coefficients satisfy
    \begin{gather}\label{cond-Cstar-I}
        \alpha\in\hlp,\quad\be\in\C.
    \end{gather}
    \item The second case is
    \begin{gather}\label{form-Cstar-II}
        g(w)=\lambda w,\quad f(w)=\theta, 
    \end{gather}
    where coefficients satisfy
    \begin{gather}\label{cond-Cstar-II}
        \lambda\in\R_{>0},\quad\theta\in\C.
    \end{gather}
    \item The third case is
    \begin{gather}\label{form-Cstar-III}
        g(w)=\delta+\dfrac{1-\delta\ell}{w+\kappa},\quad f(w)=\dfrac{r}{(w+\kappa)^{\ell+2}},
    \end{gather}
    where coefficients satisfy
    \begin{gather}\label{cond-Cstar-III}
    \begin{cases}
        \text{either $\delta\in i\R,\quad\delta\ell\in\R_{<1},\quad\kappa\in\C_{\mathbf{Re}\geq 0}$},\\
        \\
        \text{or $-\re\delta<0\leq\re\kappa-\dfrac{\re(\delta\kappa-1)+|\delta\kappa-1|}{2\re\delta}$}.
    \end{cases}
    \end{gather}
\end{enumerate}
\end{lem}
\begin{proof}
Equation \eqref{cond-Cstar-sym} becomes
\begin{gather}\label{eq-02-26-2023}
    (1+zg(w))^{\ell+2}f(z)
    =(1+wg(z))^{\ell+2}f(w),\quad\forall z,w\in\hlp.
\end{gather}
Consider two cases as follows.

Case 1: $g(\cdot)\equiv g\in\hlp$ is a constant function. In this case, \eqref{eq-02-26-2023} reveals that $(1+g\cdot)^{\ell+2}f(\cdot)$ is a complex-valued constant function. Hence, we get \eqref{form-Cstar-I}-\eqref{cond-Cstar-I}.

Case 2: $g(\cdot)$ is not a constant function. It can be found from \eqref{eq-02-26-2023} that the function $f(\cdot)$ never vanishes in $\hlp$ and moreover
\begin{gather}\label{eq-02-26-410}
    \dfrac{f(z)}{f(w)}
    =\left(\dfrac{1+wg(z)}{1+zg(w)}\right)^{\ell+2},\quad\forall z,w\in\hlp.
\end{gather}
Differentiating with respect to the variable $z$, the following is obtained
\begin{gather}
    \dfrac{f'(z)}{f(w)}
    \nonumber=(\ell+2)\left(\dfrac{1+wg(z)}{1+zg(w)}\right)^{\ell+1}\\
    \label{eq-02-26-414}\times
    \dfrac{wg'(z)(1+zg(w))-(1+wg(z))g(w)}{(1+zg(w))^2},\quad\forall z,w\in\hlp.
\end{gather}
By \eqref{eq-02-26-410} and \eqref{eq-02-26-414}, we get
\begin{gather*}
    \dfrac{f'(z)}{f(z)}
    =(\ell+2)\left(\dfrac{wg'(z)}{1+wg(z)}
    -\dfrac{g(w)}{1+zg(w)}\right),\quad\forall z,w\in\hlp,
\end{gather*}
which implies, after taking $\pa_w$, that
\begin{gather}\label{eq-02-26-832}
\dfrac{g'(z)}{(1+wg(z))^2}
=\dfrac{g'(w)}{(1+zg(w))^2},\quad\forall z,w\in\hlp.
\end{gather}
For $w_\star\in\hlp$ with $g(w_\star)\ne 0$, we have
\begin{gather*}
    \dfrac{1}{w_\star}\left(-\dfrac{1}{1+w_\star g(x)}+\dfrac{1}{1+w_\star g(w_\star)}\right)
    =\int\limits_{w_\star}^x\dfrac{g'(z)dz}{(1+w_\star g(z))^2}\\
    =\int\limits_{w_\star}^x\dfrac{g'(w_\star)dz}{(1+zg(w_\star))^2}
    =-\dfrac{g'(w_\star)}{g(w_\star)}\left(\dfrac{1}{1+xg(w_\star)}-\dfrac{1}{1+w_\star g(w_\star)}\right);
\end{gather*}
meaning that $g(\cdot)$ is a linear fractional self-map of $\hlp$. For that reason, we suppose
\begin{gather*}
    g(w)=\dfrac{Aw+B}{Cw+D},
\end{gather*}
where $A,B,C,D$ are complex coefficients.

- If $C=0$, then
\begin{gather*}
    g(w)=Ew+F,\quad\text{where $E=\dfrac{A}{D}$ and $F=\dfrac{B}{D}$}.
\end{gather*}
Through setting $g(w)=Ew+F$ in \eqref{eq-02-26-832} and then equating coefficients, we observe $F=0$ and according to \eqref{eq-02-26-2023} $f(\cdot)$ must be a complex-valued constant function. Hence, we get \eqref{form-Cstar-II}-\eqref{cond-Cstar-II}.

- If $C\ne 0$, then
\begin{gather*}
    g(w)=\dfrac{A}{C}+\left(\dfrac{BC-AD}{C}\right)\dfrac{1}{Cw+D}=E+\dfrac{F}{w+G},
\end{gather*}
where
\begin{gather*}
    E=\dfrac{A}{C},\quad F=\dfrac{BC-AD}{C^2},\quad G=\dfrac{D}{C}.
\end{gather*}
Through substituting this form of $g(\cdot)$ back into \eqref{eq-02-26-832} and then equating coefficients, we have $EG+F=1$. Hence, by \eqref{eq-02-26-2023}, $(\cdot+G)^{\ell+2}f(\cdot)$ is a complex-valued constant function. Note that condition \eqref{cond-Cstar-III} follows directly from Lemma \ref{lem-self-map}.
\end{proof}

Lemma \ref{lem-Cstar} allows us to describe weighted composition operators which are complex symmetric with respect to conjugation $\calc_\star$ (simply: \emph{$\calc_\star$-selfadjoint}).
\begin{thm}\label{thm-Cstar-1}
Let $\ell\in\Z_{\geq 0}$. Let $f:\hlp\to\C$ and $g:\hlp\to\hlp$ be analytic functions. Then the operator $W_{f,g,\max}$ is $\calc_\star$-selfadjoint on $\spc$ if and only if \eqref{form-Cstar-I}-\eqref{cond-Cstar-III} hold. In this case, the operator $W_{f,g,\max}$ is bounded.
\end{thm}
\begin{proof}
The proof is similar to those used in Theorem \ref{thm-Ca-1}.
\end{proof}

\begin{thm}\label{thm-Cstar-2}
Let $\ell\in\Z_{\geq 0}$. Let $f:\hlp\to\C$ and $g:\hlp\to\hlp$ be analytic functions. Then the operator $W_{f,g}$ is $\calc_\star$-selfadjoint on $\spc$ if and only if it verifies two conditions.
\begin{enumerate}
    \item The operator $W_{f,g}$ is maximal; that is $W_{f,g}=W_{f,g,\max}$.
    \item \eqref{form-Cstar-I}-\eqref{cond-Cstar-III} hold. 
\end{enumerate}
In this case, the operator $W_{f,g}$ is bounded.
\end{thm}
\begin{proof}
The proof is similar to thosed used in Theorem \ref{thm-Ca-2}.
\end{proof}

Let us define the linear operator $U_{\bfb,\bfc}$ by setting
\begin{gather}
    U_{\bfb,\bfc}\xi(z)=\bfb\xi\left(|\bfb|^{\frac{2}{\ell+2}}z+i\bfc\right),\quad\bfb\in\C_{\ne 0},\bfc\in\R.
\end{gather}

We first formulate an algebraic lemma.
\begin{lem}\label{lem-20210519}
Let $\ell\in\Z_{\geq 0},\bfb\in\C_{\ne 0}$ and $\bfc\in\R$. Then the following assertions hold.
\begin{enumerate}
    \item The operator $U_{\bfb,\bfc}$ is unitary on $\spc$ and moreover
\begin{gather}
    U_{\bfb,\bfc}^*\xi(z)=\dfrac{1}{\bfb}\xi\left(|\bfb|^{-\frac{2}{\ell+2}}(z-i\bfc)\right).
\end{gather}
\item Always have
\begin{gather*}
    U_{\bfb,\bfc}^*\bfe_{f,g}U_{\bfb,\bfc}=\bfe_{\widehat{f},\widehat{g}},
\end{gather*}
where
\begin{gather*}
    \widehat{g}(z)=|\bfb|^{\frac{2}{\ell+2}}g\left(|\bfb|^{-\frac{2}{\ell+2}}(z-i\bfc)\right)+i\bfc,
    \quad\widehat{f}(z)=f\left(|\bfb|^{-\frac{2}{\ell+2}}(z-i\bfc)\right).
\end{gather*}
\end{enumerate}
\end{lem}

Our study is motivated by the following lemma.
\begin{lem}\label{lem-UCU*-conj}
Let $\ell\in\Z_{\geq 0},\bfb\in\C_{\ne 0}$ and $\bfc\in\R$. Then the operator $U_{\bfb,\bfc}\calc_{\star}U_{\bfb,\bfc}^*$ is a conjugation on $\spc$.
\end{lem}
The proofs of Lemmas \ref{lem-20210519} and \ref{lem-UCU*-conj} need a few computation steps and they are left to the reader. Lemma \ref{lem-UCU*-conj} leads to describe weighted composition operators which are complex symmetric with respect to conjugation $U_{\bfb,\bfc}\calc_{\star}U_{\bfb,\bfc}^*$ (or simply: \emph{$U_{\bfb,\bfc}\calc_{\star}U_{\bfb,\bfc}^*$-selfadjoint}).
\begin{thm}\label{thm-UCstarU-1}
Let $\ell\in\Z_{\geq 0},\bfb\in\C_{\ne 0}$ and $\bfc\in\R$. Let $f:\hlp\to\C$ and $g:\hlp\to\hlp$ be analytic functions. Then the following assertions are equivalent.
\begin{enumerate}
    \item The operator $W_{f,g,\max}$ is $U_{\bfb,\bfc}\calc_{\star}U_{\bfb,\bfc}^*$-selfadjoint on $\spc$.
    \item The operator $W_{\widehat{f},\widehat{g},\max}=U_{\bfb,\bfc}^*W_{f,g,\max}U_{\bfb,\bfc}$ is $\calc_{\star}$-selfadjoint on $\spc$.
    \item There are three cases of $f(\cdot)$ and $g(\cdot)$.
    \begin{enumerate}
        \item The first case is
        \begin{gather}\label{form-20210312-I}
            g(w)=|\bfb|^{-\frac{2}{\ell+2}}(\alpha-i\bfc),
            \quad f(w)=\dfrac{\be}{\left(1+\alpha(|\bfb|^{\frac{2}{\ell+2}}w+i\bfc)\right)^{\ell+2}},
        \end{gather}
        where $\alpha,\be$ satisfy \eqref{cond-Cstar-I}.
        \item The second case is
        \begin{gather}\label{form-20210312-II}
            g(w)=\lambda\left(w+i\bfc|\bfb|^{-\frac{2}{\ell+2}}\right)-i\bfc,
            \quad f(w)=\theta,
        \end{gather}
        where $\lambda,\theta$ satisfy \eqref{cond-Cstar-II}.
        \item The third case is
        \begin{gather}
            \nonumber g(w)=|\bfb|^{-\frac{2}{\ell+2}}\left(\dfrac{1-\delta\kappa}{|\bfb|^{\frac{2}{\ell+2}}w+i\bfc+\kappa}+\delta-i\bfc\right),
            \\\label{form-20210312-III}f(w)=\dfrac{r}{\left(|\bfb|^{\frac{2}{\ell+2}}w+i\bfc+\kappa\right)^{\ell+2}},
        \end{gather}
        where $\delta,\kappa,r$ satisfy \eqref{cond-Cstar-III}.
    \end{enumerate}
\end{enumerate}
In this case, the operator $W_{f,g,\max}$ is bounded.
\end{thm}
\begin{proof}
    The proof makes use of Theorem \ref{thm-Cstar-1} and Lemma \ref{lem-20210519}.
\end{proof}

\begin{thm}\label{thm-UCstarU-2}
Let $\ell\in\Z_{\geq 0},\bfb\in\C_{\ne 0},\bfc$ and $\in\R$. Let $f:\hlp\to\C$ and $g:\hlp\to\hlp$ be analytic functions. Then the operator $W_{f,g}$ is $U_{\bfb,\bfc}\calc_{\star}U_{\bfb,\bfc}^*$-selfadjoint on $\spc$ if and only if it verifies two conditions.
\begin{enumerate}
\item The operator $W_{f,g}$ is maximal; that is $W_{f,g}=W_{f,g,\max}$.
\item \eqref{form-20210312-I}-\eqref{form-20210312-III} hold. 
\end{enumerate}
In this case, the operator $W_{f,g}$ is bounded.
\end{thm}
\begin{proof}
    The proof makes use of Theorem \ref{thm-Cstar-2} and Lemma \ref{lem-20210519}.
\end{proof}

\section{Some corollaries}\label{sec-cor}
A question to study is how big is the class of $\calc_\bfa$-selfadjoint, and $U_{\bfb,\bfc}\calc_{\star}U_{\bfb,\bfc}^*$-selfadjoint operators. In the following result, we show that this class is very interesting; namely, it contains properly hermitian operators studied in Section \ref{sec-her} and unitary operators investigated in Section \ref{sec-uni-op}.
\begin{cor}
Let $\ell\in\Z_{\geq 0}$. Let $f:\hlp\to\C$ and $g:\hlp\to\hlp$ be analytic functions. Then the following assertions hold.
\begin{enumerate}
    \item If the operator $W_{f,g}$ is hermitian, then it is $\calc_\bfa$-selfadjoint for some $\bfa$.
    \item Suppose that the operator $W_{f,g,\max}$ is unitary (that is \eqref{form-uni-2} or \eqref{form-uni} holds).
    \begin{enumerate}
        \item If the functions are of forms in \eqref{form-uni-2}, then the operator $W_{f,g,\max}$ is $U_{\bfb,\bfc}\calc_{\star}U_{\bfb,\bfc}^*$-selfadjoint for some $\bfb,\bfc$.
        \item If the functions are of forms in \eqref{form-uni}, then the operator $W_{f,g,\max}$ is $\calc_{\bfa}$-selfadjoint for some $\bfa$.
    \end{enumerate}
\end{enumerate}
\end{cor}
\begin{proof}
(1) Suppose that the operator $W_{f,g}$ is hermitian. By Theorem \ref{thm-her-2}, it must be bounded and there are three cases of functions $f(\cdot)$ and $g(\cdot)$.

- If the functions $f(\cdot)$ and $g(\cdot)$ verify \eqref{form-a-b-del-her-III}-\eqref{cond-a-b-del-her-III}, then we take $\bfa=2\im\mu$. For this choice, \eqref{form-cs-Ca-1-I}-\eqref{cond-cs-Ca-I} hold and so by Theorem \ref{thm-Ca-1}, the operator $W_{f,g}$ is $\calc_\bfa$-selfadjoint.

- If the functions $f(\cdot)$ and $g(\cdot)$ verify \eqref{form-a-b-del-her-II}-\eqref{cond-a-b-del-her-II}, then \eqref{form-cs-Ca-1-II}-\eqref{cond-cs-Ca-II} hold automatically.

- If the functions $f(\cdot)$ and $g(\cdot)$ verify \eqref{form-her}-\eqref{cond-self-map-0}, then \eqref{form-cs-Ca-1}-\eqref{cond-self-map-cs-Ca} hold for choosing $\bfa=\frac{2\im\beta}{\alpha}$. Thus, the hermitian operator $W_{f,g}$ is $\calc_\bfa$-selfadjoint for some $\bfa$.

(2) Suppose that the operator $W_{f,g,\max}$ is unitary. By Theorem \ref{thm-uni}, there are two cases. If the functions are of forms in \eqref{form-uni-2}, then they can be expressed in the forms in \eqref{form-20210312-II}. Hence, by Theorem \ref{thm-UCstarU-2}, the operator $W_{f,g,\max}$ is $U_{\bfb,\bfc}\calc_{\star}U_{\bfb,\bfc}^*$-selfadjoint for some $\bfb,\bfc$. If the functions are of forms in \eqref{form-uni}, then \eqref{form-cs-Ca-1} holds. So, by Theorem \ref{thm-Ca-2}, the operator $W_{f,g,\max}$ is $\calc_{\bfa}$-selfadjoint for some $\bfa$.
\end{proof}

\section {Composition operators}\label{sec-com-op}
Next, we focus only on the very restrictive category, that is bounded composition operators induced by linear fractional transforms. It was proven in \cite{zbMATH05907249} that such operators have an explicit form.
\begin{prop}[{\cite{zbMATH05907249}}]\label{prop-bdd-lff}
The linear fractional function of $\hlp$ inducing a bounded composition operator $C_{g,\max}$ on $\spc$ is of form 
\begin{align}\label{1}
g(w)=\mu w+w_0, \quad \text{where} \ \mu\in\R_{>0} \ \text{and} \ w_0\in\C_{\mathbf{Re}\geq 0}.
\end{align}
\end{prop}

\subsection{Complex symmetry}
The results obtained in the previous sections give some information about the complex symmetry of composition operators induced by functions of form \eqref{1}.
\begin{prop}
The following conclusions hold.
    \begin{enumerate}
        \item If $g(w)=w+w_0$, where $w_0\in\C_{\mathbf{Re}\geq 0}$, then the operator $C_{g,\max}$ is $\calc_\bfa$-symmetric for every $\bfa\in\R$.
        \item If $g(w)=\mu w$, where $\mu\in\R_{>0}$, then the operator $C_{g,\max}$ is $\calc_\star$-symmetric.
        \item If $g(w)=\mu w+w_0$, where $\mu\in\R_{>0}\setminus\{1\}$ and $w_0\in i\R$, then the operator $C_{g,\max}$ is $U_{1,\bfc}^*\calc_\star U_{1,\bfc}$-symmetric, where $\bfc=i(1-\mu)^{-1}w_0$.
    \end{enumerate}
\end{prop}
\begin{proof}
    The first conclusion is a consequence of Theorem \ref{thm-Ca-1}. The second is obtained from Theorem \ref{thm-Cstar-1} and the last follows from Theorem \ref{thm-UCstarU-1}.
\end{proof}
The remaining task is to consider the case when 
\begin{gather}\label{cond-20210401}
    \mu\in\R_{>0}\setminus\{1\},\quad\text{and}\,\,w_0\in\hlp. 
\end{gather}
As it turns out, this case is in connection with the fixed points of the function $g(\cdot)$. We pause for a while to recall some terminologies. Let $\phi^{[n]}(\cdot)$ denote the $n$-th iterate of the self-map $\phi(\cdot).$ If
$\omega$ is a point of the closure of the open unit disk $\mathbb{D}$ such that the sequence of iterates $\phi^{[n]}:\mathbb{D}\longrightarrow \mathbb{D}$ converges uniformly on compacts subsets of $\mathbb{D}$ to $\omega,$ then $\omega$ is said to be an \textit{attractive point} for $\phi(\cdot).$ The Denjoy-Wolff Theorem states that if $\phi(\cdot)$ is an analytic self-map of $\mathbb{D}$ is not an elliptic automorphism then there is an unique point in $\omega\in\overline{\mathbb{D}}$ such that $\phi^{[n]}(z)\longrightarrow \omega$ as $n \rightarrow\infty,$ for each $z\in \mathbb{D}$ (see \cite[Theorem 2.51]{zbMATH00918588}). For analytic self-maps of $\hlp$ inducing bounded composition operators on $\spc,$ we have the following version:
\begin{thm}\label{13}
Let $g(\cdot)$ be an analytic self-map of $\hlp$ such that $C_{g,\max}$ is bounded. If $g(\cdot)$ has a fixed point $\alpha \in \hlp$ then 
\begin{align*}
\alpha=\displaystyle \lim_{n\longrightarrow \infty}g^{[n]}(w),
\end{align*} 
for each $w\in \hlp.$	
\end{thm}
\begin{proof}
Let $\gamma(z)=\frac{1-z}{1+z}$ be the \textit{Mobius transform} of $\mathbb{D}$ onto $\hlp.$ Then $\Psi=\gamma^{-1}\circ g \circ \gamma$ is an analytic self-map of $\mathbb{D}$ whose fixed point is $\gamma^{-1}(\alpha).$ Now observe that $g(\cdot)$ is not an automorphism. Indeed, the only automorphism of $\hlp$ that induce bounded composition operators are $g(w)=w+w_0$ where $\mathrm{Re}(w_0)\geq 0$	and $g(w)=\mu w+\mathrm{i}r$ where $\mu\in (0,1)\cup(1,\infty)$ and $r\in \mathbb{R},$ and both cases $g(\cdot)$ not have fixed points in $\hlp.$ This implies that $\Psi(\cdot)$ also is not an automorphism of $\mathbb{D}.$ By the Denjoy-Wolff Theorem, the iterates $\Psi^{[n]}(\cdot)\longrightarrow \gamma^{-1}(\alpha)$ locally uniformly in $\mathbb{D}$ as $n \rightarrow\infty,$ and hence $g^{[n]}(\cdot)\longrightarrow \alpha$ locally uniformly in $\hlp$ as $n \rightarrow\infty.$
\end{proof}

\begin{thm}\label{7}
Let $g(\cdot)$ be an analytic self-map of $\hlp$ such that $C_{g,\max}$ is bounded on $\spc.$ If $g(\cdot)$ has a fixed point in $ \hlp$, then $C_{g,\max}$ is not complex symmetric.
\end{thm}
\begin{proof} 
Assume in contrary that the operator $C_{g,\max}$ is complex symmetric. Let $C$ be a conjugation on $\spc$ such that $CC_{g,\max}C=C_{g,\max}^*.$ If there is $\alpha\in \hlp$ such that $g(\alpha)=\alpha,$ then Proposition \ref{W*Kz-prop} gives $C_{g,\max}^*K_{\alpha}=K_{\alpha}.$ Now observe that
\begin{align*}C_{g,\max}CK_{\alpha}=CC_{g,\max}^*K_{\alpha}=CK_{\alpha},
\end{align*}
which implies, with note that $C_{g^{[n]},\max}=C_{g,\max}^n$, that $(CK_{\alpha})( g^{[n]}(w)) =(CK_{\alpha})(w),$ for $w\in \hlp.$ By Theorem \ref{13}, $g^{[n]}(\cdot)$ converges to $\alpha$ locally uniformly in $\hlp$ as $n \rightarrow\infty.$ Then $CK_{\alpha}= (CK_{\alpha})(\alpha).$ Since the only constant function on $\spc$ is the function identically zero, we have $CK_{\alpha}= 0$ and hence $K_{\alpha}=C^2K_{\alpha}=0$; but this is impossible.
\end{proof}	

\begin{cor}\label{8}
If $g(w)=\mu w+w_0$, where coefficients satisfy \eqref{cond-20210401}, then the operator $C_{g,\max}$ is never complex symmetric.
\end{cor}
\begin{proof} The proof follows directly from Theorem \ref{7}. 
\end{proof}

\subsection{Normality}\label{4}
In this section, we characterize which linear fractional composition operators are normal, self-adjoint, unitary and isometric. Our first step is to establish a formula for the adjoint of linear fractional composition operators on $\spc.$ 

\begin{prop}\label{3} If the function $g(\cdot)$ is of form \eqref{1}, then the adjoint of $C_{g,\max}$ on $\spc$ is given by
\begin{align*}
C_{g,\max}^{*}=\mu^{-(\ell+2)}C_{g_\star,\max},
\end{align*}
where $g_\star(w)=\mu^{-1}w+\mu^{-1}\overline{w_0}.$
\end{prop}
\begin{proof} We first observe that for $z,w\in \hlp,$ we have
\begin{align*}
(C_{g,\max}K_z)(w)=\frac{2^{\ell}(1+\ell)}{\mu^{\ell+2} \left( \mu^{-1}\overline{z}+\mu^{-1} w_0+w\right)^{\ell+2} }
=\left( \mu^{-(\ell+2)} K_{g_\star(z)}\right) (w).
\end{align*}
Then given $h(\cdot)\in \spc,$ we get
\begin{align*}
\left( C_{g,\max}^{*}h\right) (z)= \langle C_{g,\max}^*h, K_z\rangle= \langle h,C_{g,\max}K_z\rangle\\
=  \mu^{-(\ell+2)}\langle h,K_{g_\star(z)}\rangle = \left( \mu^{-(\ell+2)}C_{g_\star
	,\max}h\right) (z)
\end{align*}
which implies that $C_{g,\max}^*=\mu^{-(\ell+2)}C_{g_\star,\max}.$
\end{proof}

Using Proposition \ref{3}, we can prove the following result.
\begin{thm}\label{2}
Let $g(w)=\mu w+w_0$ with $\mu>0$ and $\mathrm{Re}(w_0)\geq 0.$ Then
\begin{enumerate}
\item $C_{g,\max}$ is normal on $\spc$ if and only if $\mu=1$ or $\mathrm{Re}(w_0)=0.$ 
\item $C_{g,\max}$ is self-adjoint on $\spc$ if and only if $\mu=1$ and $\mathrm{Re}(w_0)\geq 0.$
\item $C_{g,\max}$ is unitary on $\spc$ if and only if $\mu=1$ and $\mathrm{Re}(w_0)=0.$
\item $C_{g,\max}$ is isometric on $\spc$ if and only if it is unitary on $\spc.$
\end{enumerate}
\end{thm}
\begin{proof}  
We omit the proofs for item (1-3), and prove item (4). Since each unitary operator is isometric, it is enough to prove that if $C_{g,\max}$ is isometric then $C_{g,\max}$ is unitary. Suppose that the operator $C_{g,\max}$ is isometric on $\spc$ then $\|C_{g,\max}K_z\|^2=\|K_z\|^2,$ for each $z\in \hlp.$ Now observe that 
\begin{align}\label{12}
\left\| K_z\right\| ^2=\left\langle K_{z},K_{z}\right\rangle =K_{z}(z)=\frac{2^{\ell}(1+\ell)}{2^{\ell+2}\mathrm{Re}(z)^{\ell+2}}.
\end{align}
Moreover, combining Propositions \ref{W*Kz-prop} and \ref{3}, we obtain $C_{g,\max}K_{z}=\mu^{-(\ell+2)}K_{g_\star(z)}$ where $g_\star(w)=\mu^{-1}w+\mu^{-1}\overline{w_0}.$  Then \eqref{12} gives
\begin{align*}
\left\| C_{g,\max}K_z\right\| ^2=\frac{2^{\ell}(1+\ell)\mu^{-2(\ell+2)}}{2^{\ell+2}\mathrm{Re}(g_\star(z))^{\ell+2}}=\frac{2^{\ell}(1+\ell)\mu^{-(\ell+2)}}{2^{\ell+2}\mathrm{Re}(z+w_0)^{\ell+2}},
\end{align*}
and hence 
\begin{align*}
\frac{2^{\ell}(1+\ell)}{2^{\ell+2}\mathrm{Re}(z)^{\ell+2}}
=\frac{2^{\ell}(1+\ell)\mu^{-(\ell+2)}}{2^{\ell+2}\mathrm{Re}(z+w_0)^{\ell+2}},
\end{align*}
or equivalently $\mathrm{Re}(z+w_0)=\mu^{-(\ell+2)}\mathrm{Re}(z).$ Choosing $z=1$ and $z=2,$ we obtain the following system
\begin{align*}
\left\lbrace \begin{array}{rl}
1+\mathrm{Re}(w_0)&=\mu^{-(\ell+2)}\\
2+\mathrm{Re}(w_0)&=2\mu^{-(\ell+2)}
\end{array}\right. 
\end{align*}
whose solution is $\mu=1$ and $\mathrm{Re}(w_0)=0.$ By item (3), the operator $C_{g,\max}$ is unitary.
\end{proof}
 For general case, Theorem \ref{7} provides the following result:
 \begin{thm} $\spc$ not support normal composition operators whose function $g(\cdot)$ has a fixed point in $\hlp.$
 \end{thm}	

\section{A natural link to complex symmetry in Lebesgue spaces}\label{sec-link-Lebes}
Lebesgue space $\lb$ consists of measurable functions $h:\R_{>0}\to\C$ for which
\begin{gather*}
\|\bfh\|^2
=\int\limits_0^\infty|\bfh(t)|^2\dfrac{\Gamma(1+\ell)}{2^\ell t^{1+\ell}}\,dt<\infty.
\end{gather*}
Our research is motivated by a Paley-Wiener theorem, which states that Bergman space $\spc$ is isometrically isomorphic under the Laplace transform to Lebesgue space $\lb$. In fact, to each function $h(\cdot)\in\spc$, there corresponds a function $\bfh(\cdot)\in\lb$ such that $h(\cdot)=\lp(\bfh)(\cdot)$, where the symbol $\lp$ stands for the Laplace transform
\begin{gather*}
\lp(\bfh)(z)=\int\limits_0^\infty\bfh(t)e^{izt}\,dt.
\end{gather*}

The following example lists some Laplace transform formulas.
\begin{exa}[\cite{zbMATH00432614}]\label{exa}
	For every $z\in\hlp$, we always have
	\begin{gather*}
	K_z=\lp\left(\dfrac{2^\ell}{\ell!}t^{\ell+1}e^{-t\overline{z}}\right).
	\end{gather*}
\end{exa}

In the following result, we study the transformation of conjugations via the Laplace transform.
\begin{prop}
	\begin{enumerate}
		\item The operator $\lp^{-1}\calc_\bfa\lp$ is a conjugation on $\lb$ and
		\begin{gather*}
		\lp^{-1}\calc_\bfa\lp\left(\dfrac{2^\ell}{\ell!}t^{\ell+1}e^{-t\overline{z}}\right)
		=\dfrac{2^\ell}{\ell!}t^{\ell+1}e^{-t(z-i\bfa)},\quad z\in\hlp.
		\end{gather*}
		In particular with $\bfa=0$, we have the explicit form
		\begin{gather*}
		\left(\lp^{-1}\calj\lp\right)\bfh(t)=\overline{\bfh(t)},\quad\forall\bfh(\cdot)\in\lb.
		\end{gather*}
		\item The operator $\lp^{-1}\calc_\star\lp$ is a conjugation on $\lb$ and
		\begin{gather*}
		\lp^{-1}\calc_\star\lp\left(\dfrac{2^\ell}{\ell!}t^{\ell+1}e^{-t\overline{z}}\right)
		=\dfrac{2^\ell}{z^{\ell+2}\ell!}t^{\ell+1}e^{-\frac{t}{z}},\quad z\in\hlp.
		\end{gather*}
	\end{enumerate}
\end{prop}
\begin{proof}
	(1) For $z\in\hlp$, we have
	\begin{gather*}
	\lp^{-1}\calc_\bfa\lp\left(\dfrac{2^\ell}{\ell!}t^{\ell+1}e^{-t\overline{z}}\right)
	=\lp^{-1}\calc_\bfa K_z\quad\text{(by Example \ref{exa})}\\
	=\lp^{-1}K_{\overline{z}+i\bfa}\quad\text{(by Lemma \ref{lem-cs-Ca})}\\
	=\dfrac{2^\ell}{\ell!}t^{\ell+1}e^{-t(z-i\bfa)}\quad\text{(by Example \ref{exa})}.
	\end{gather*}
	
	(2) The proof is left to the reader as it is a direct computation.
\end{proof}

The result below shows that a composition operator on Bergman space can be transformed into a weighted composition operator on Lebesgue space via the Laplace transform. Moreover, this transformation preserves the complex symmetry.
\begin{prop}
	Let $f(\cdot)$ and $g(\cdot)$ be the functions in \eqref{form-20210312-II} with \eqref{cond-Cstar-II}. If
	\begin{gather*}
	\psi(t)=\dfrac{\theta}{\lambda^{\ell+2}}e^{-ti\bfc(\lambda|\bfb|^{-\frac{2}{\ell+2}}-1)/\lambda} \quad \text{and}\quad
	\quad\varphi(t)=\dfrac{t}{\lambda}
	\end{gather*}
	then the following assertions hold.
	\begin{enumerate}
		\item $\lp^{-1}\bfe_{f,g}\lp=\bfe_{\psi,\varphi}$.
		\item Furthermore, let $W_{f,g,\max}$ be the maximal weighted composition operator generated by $\bfe_{f,g}$ on Bergman space $\spc$ and $W_{\psi,\varphi,\max}$ be the maximal weighted composition operator generated by $\bfe_{\psi,\varphi}$ on Lebesgue space $\lb$. Then the operator
		\begin{gather*}
		W_{\psi,\varphi,\max}=\lp^{-1}W_{f,g,\max}\lp
		\end{gather*}
		is $\lp^{-1}U_{\bfb,\bfc}\calc_{\star}U_{\bfb,\bfc}^*\lp$-selfadjoint on Lebesgue space $\lb$.
	\end{enumerate}
\end{prop}
\begin{proof}
	Note that the second conclusion follows directly from the first. Now we prove the first conclusion as follows. Denote $d=i\bfc(\lambda|\bfb|^{-\frac{2}{\ell+2}}-1)$. For $z\in\hlp$, we have
	\begin{gather*}
	\lp^{-1}\bfe_{f,g}\lp\left(\dfrac{2^\ell}{\ell!}t^{\ell+1}e^{-t\overline{z}}\right)
	=\lp^{-1}\bfe_{f,g}K_z\quad\text{(by Example \ref{exa})}\\
	=\dfrac{\theta}{\lambda^{\ell+2}}\lp^{-1}K_{(z-d)/\lambda}
	=\dfrac{\theta}{\lambda^{\ell+2}}e^{-td/\lambda}
	\dfrac{2^\ell}{\ell!}t^{\ell+1}e^{-t\overline{z}/\lambda}\quad\text{(by Example \ref{exa})}.
	\end{gather*}
\end{proof}

\section*{Acknowledgements}

O.R. Severiano is postdoctoral fellowship at Programa Associado de P\'{o}s Gradua\c{c}\~{a}o em Matem\'{a}tica  UFPB/UFCG, and is supported by INCTMat Grant 88887.613486/2021-00.

\section*{Data Statement}
The research does not include any data.

\bibliographystyle{plain}
\bibliography{refs}
\end{document}